\documentclass[12pt]{amsart}

\usepackage{amsmath, amssymb}
\usepackage{graphicx}

\newtheorem{theorem}{Theorem}[section]
\newtheorem{lemma}[theorem]{Lemma}

\theoremstyle{definition}

\theoremstyle{remark}
\newtheorem{remark}[theorem]{Remark}

\numberwithin{equation}{section}

\newcommand{\rr}{{\mathbb R}}
\newcommand{\rd}{{\mathbb R^d}}

\newcommand{\nat}{{\mathbb N}}
\newcommand{\ganz}{{\mathbb Z}}
\newcommand{\Exp}{{\mathbb E}}

\newcommand{\Ima}{\operatorname{Im}}
\newcommand{\card}{\operatorname{card}}
\newcommand{\LIM}{\operatorname{LIM}}
\newcommand{\eqd}{\stackrel{\rm d}{=}}
\DeclareMathOperator{\diml}{\underline{dim}}
\DeclareMathOperator{\dimu}{\overline{dim}}
\allowdisplaybreaks

\begin{document}

\sloppy
\title[The dimension of the St.\ Petersburg game]{The dimension of the St.\ Petersburg game} 
\author{Peter Kern}
\address{Peter Kern, Mathematisches Institut, Heinrich-Heine-Universit\"at D\"usseldorf, Universit\"atsstr. 1, D-40225 D\"usseldorf, Germany}
\email{kern\@@{}math.uni-duesseldorf.de}

\author{Lina Wedrich}
\address{Lina Wedrich, Mathematisches Institut, Heinrich-Heine-Universit\"at D\"usseldorf, Universit\"atsstr. 1, D-40225 D\"usseldorf, Germany}
\email{wedrich\@@{}math.uni-duesseldorf.de} 

\date{\today}

\begin{abstract}
Let $S_n$ be the total gain in $n$ repeated St.\ Petersburg games. It is known that $n^{-1}(S_n-n\log_2n)$ converges in distribution to a random element $Y(t)$ along subsequences of the form $k(n)=2^{p(n)}t(n)$ with $p(n)=\lceil\log_2k(n)\rceil\to\infty$ and $t(n)\to t\in[\frac12,1]$. We determine the Hausdorff and box-counting dimension of the range and the graph for almost all sample paths of the stochastic process $\{Y(t)\}_{t\in[1/2,1]}$. The results are compared to the fractal dimension of the corresponding limiting objects when gains are given by a deterministic sequence initiated by Hugo Steinhaus.
\end{abstract}


\maketitle

\baselineskip=18pt

\section{Introduction}

The famous St.\ Petersburg game is easily formulated as a simple coin tossing game. The player's gain $Y=2^T$ in a single game can be expressed by means of the stopping time $T=\inf\{n\in\nat:\,X_n=1\}$ of repeated independent tosses $(X_n)_{n\in\nat}$ of a fair coin until it first lands heads. For a sequence of gains $(Y_n)_{n\in\nat}$ in independent St.\ Petersburg games the partial sum $S_n=\sum_{k=1}^nY_k$ denotes the total gain in the first $n$ games. To find a fair entrance fee for playing the game is commonly called the St.\ Petersburg problem, frequently raised to the status of a paradox. Since the expectation $\Exp[Y]=\infty$ is infinite, a fair premium cannot be constructed by the help of the usual law of large numbers. We refer to Jorland \cite{Jor} and Dutka \cite{Dut} for the history of the St.\ Petersburg game and for early solutions of the 300 year old problem. 

The first step towards a mathematically satisfactory solution has been achieved by Feller \cite{Fel1,Fel2} who showed that a time-dependent premium can fulfill a certain weak law of large numbers
$$\frac{S_n}{n\log_2n}\to1\quad\text{ in probability,}$$
where $\log_2$ denotes logarithm to the base $2$. However, Feller's result does not tell if the game is dis- or advantageous for the player, i.e.\ if $S_n-n\log_2n$ is likely to be negative or positive. This question can only be answered by a weak limit theorem and the first theorem of this kind has been  shown by Martin-L\"of \cite{Mar} for the subsequence $k(n)=2^n$
$$\frac{S_{k(n)}-k(n)\log_2k(n)}{k(n)}\to X\quad\text{ in distribution.}$$
The limit $X$ is infinitely divisible with characteristic function $\exp(\psi(y))$, where
$$\psi(y)=\int_{0+}^\infty e^{iyx}-1-iyx\cdot 1_{\{x\leq1\}}\,d\phi(x)$$
and the L\'evy measure $\phi$ is concentrated on $2^\ganz$ with $\phi(\{2^k\})=2^{-k}$ for $k\in\ganz$. Hence $X$ is a semistable random variable and the corresponding L\'evy process $\{X(t)\}_{t\geq0}$ with $X(1)\eqd X$ is a (non-strictly) semistable L\'evy process fulfilling the semi-selfsimilarity condition
$$X(2^kt)\eqd2^k(X(t)+kt)\quad\text{ for every }k\in\ganz\text{ and }t\geq0.$$
For details on semistable random variables and L\'evy processes we refer to the monographs \cite{MS,Sat}. The nature of semistability is that there exists in general a continuum of possible limit distributions. For the St.\ Petersburg game the possible limit distributions have been characterized by Cs\"org\H{o} and Dodunekova \cite{CD} who proved that for any subsequence $k(n)\to\infty$ with
$$2^{-\lceil\log_2k(n)\rceil}k(n) \to t\in[\tfrac12,1]$$
we have
$$\frac{S_{k(n)}-k(n)\log_2k(n)}{k(n)}\to Y(t)=t^{-1}(X(t)-t\log_2t)\quad\text{ in distribution,}$$
where $Y(\frac12)\eqd Y(1)\eqd X$; cf.\ also \cite{Cso,Var1}.

The object of our study are local fluctuations of the sample paths of the stochastic process $Y=\{Y(t)\}_{t\in[\frac12,1]}$ consisting of all the possible weak limits of normalized total gains in repeated St.\ Petersburg games. Figure \ref{pb} shows typical (approximative) sample paths of $\{Y(t)\}_{t\in[\frac12,1]}$ generated by $n=2^{16}$ simulated St.\ Petersburg games.
\begin{figure}[h]
\includegraphics[scale=0.3]{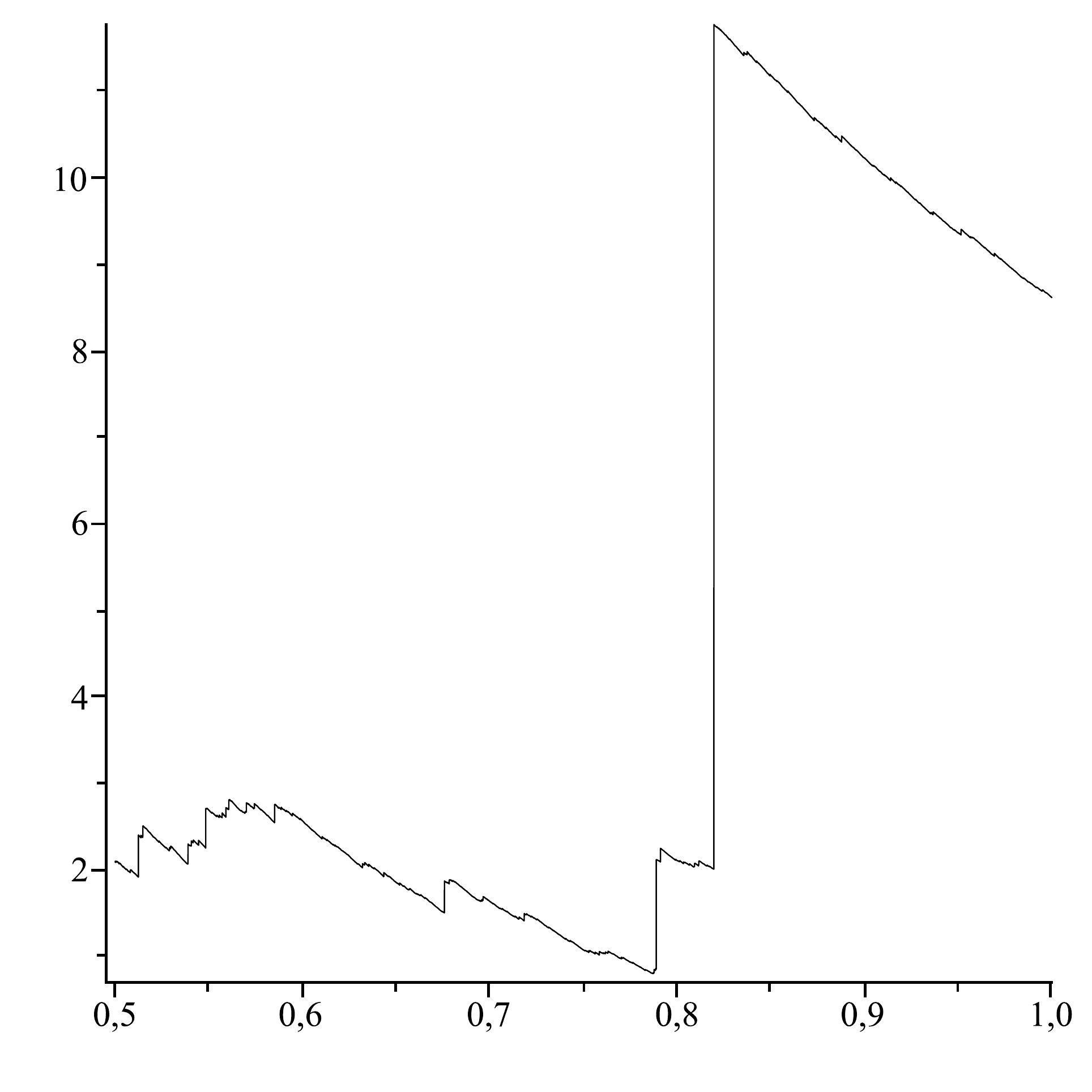}
\includegraphics[scale=0.3]{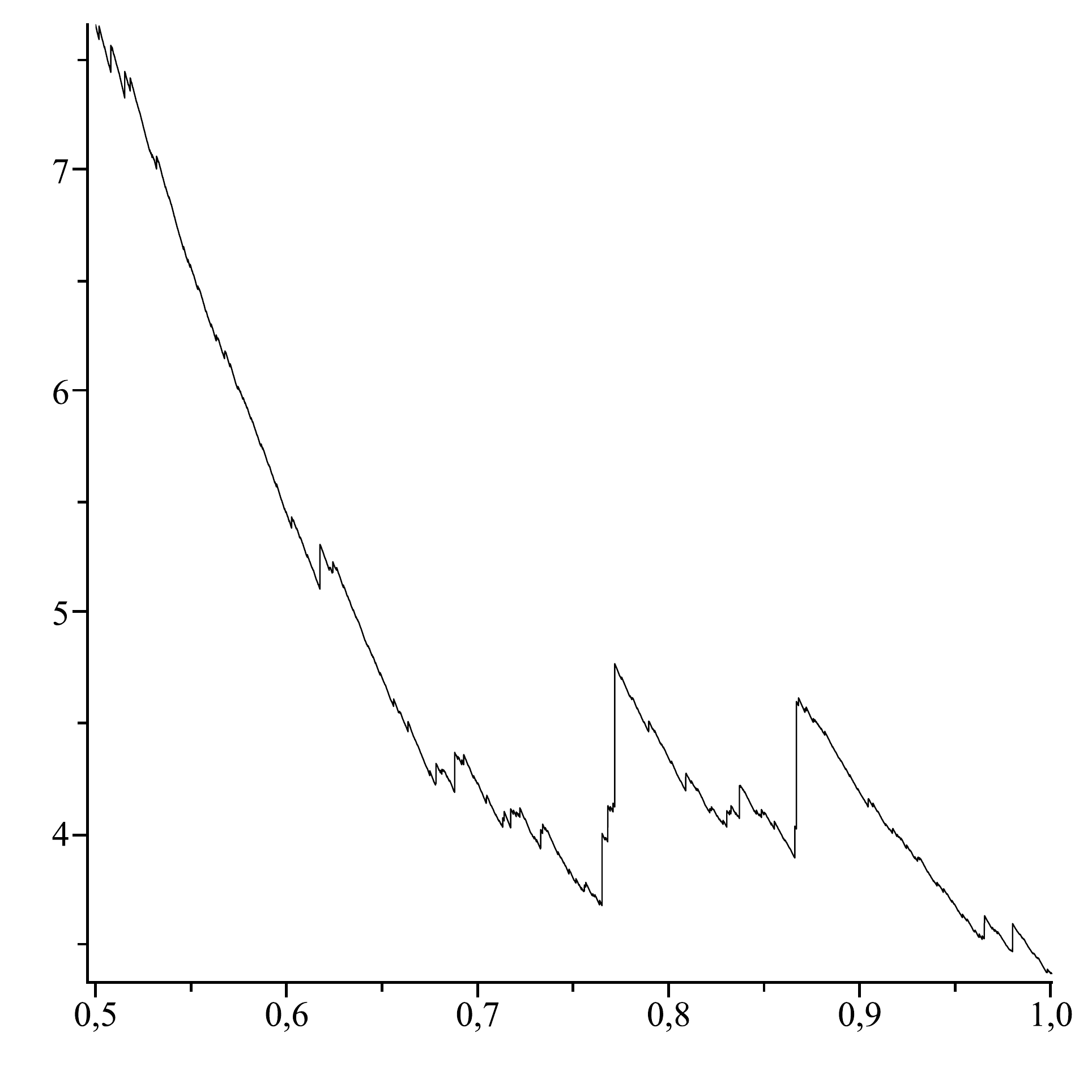}
\includegraphics[scale=0.3]{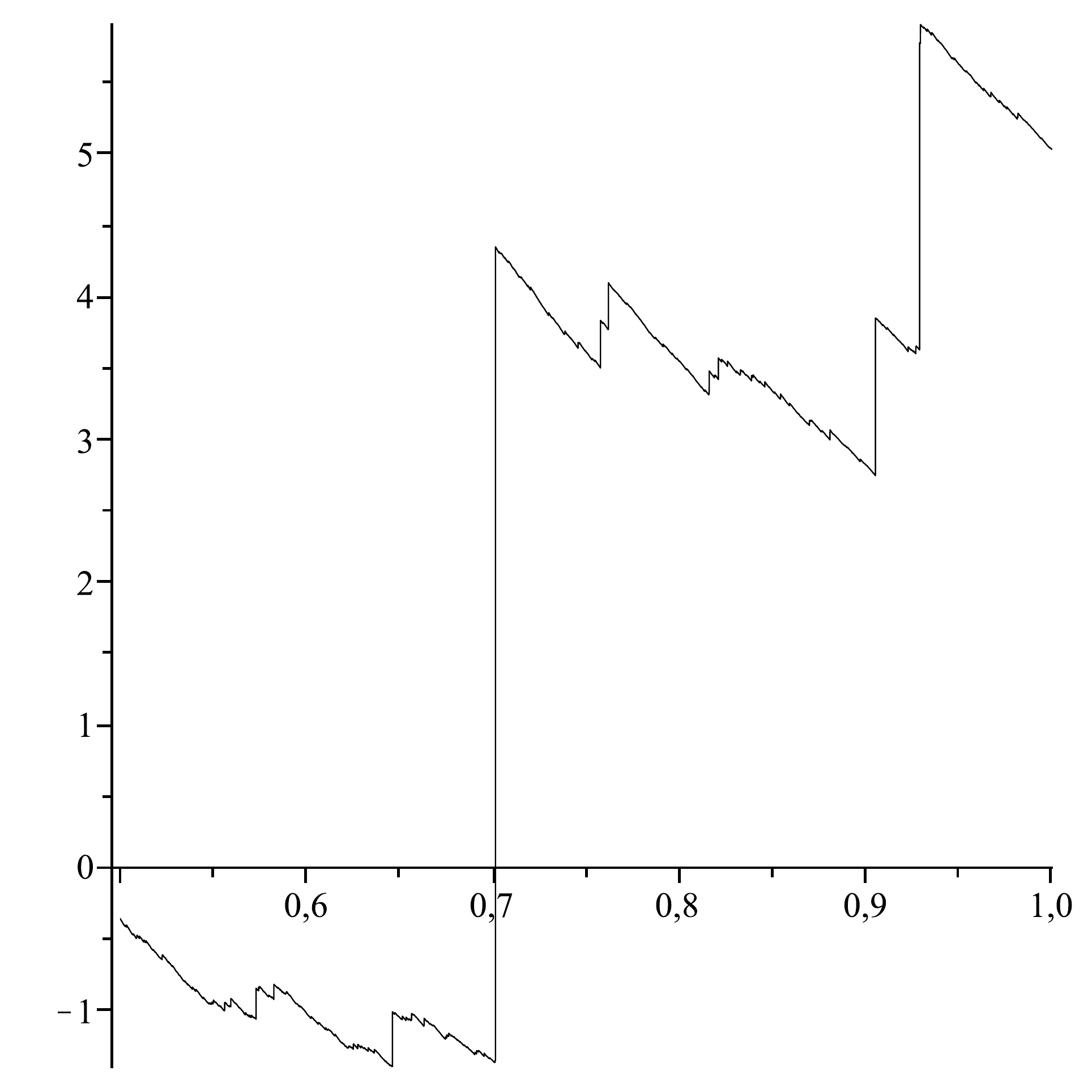}
\includegraphics[scale=0.3]{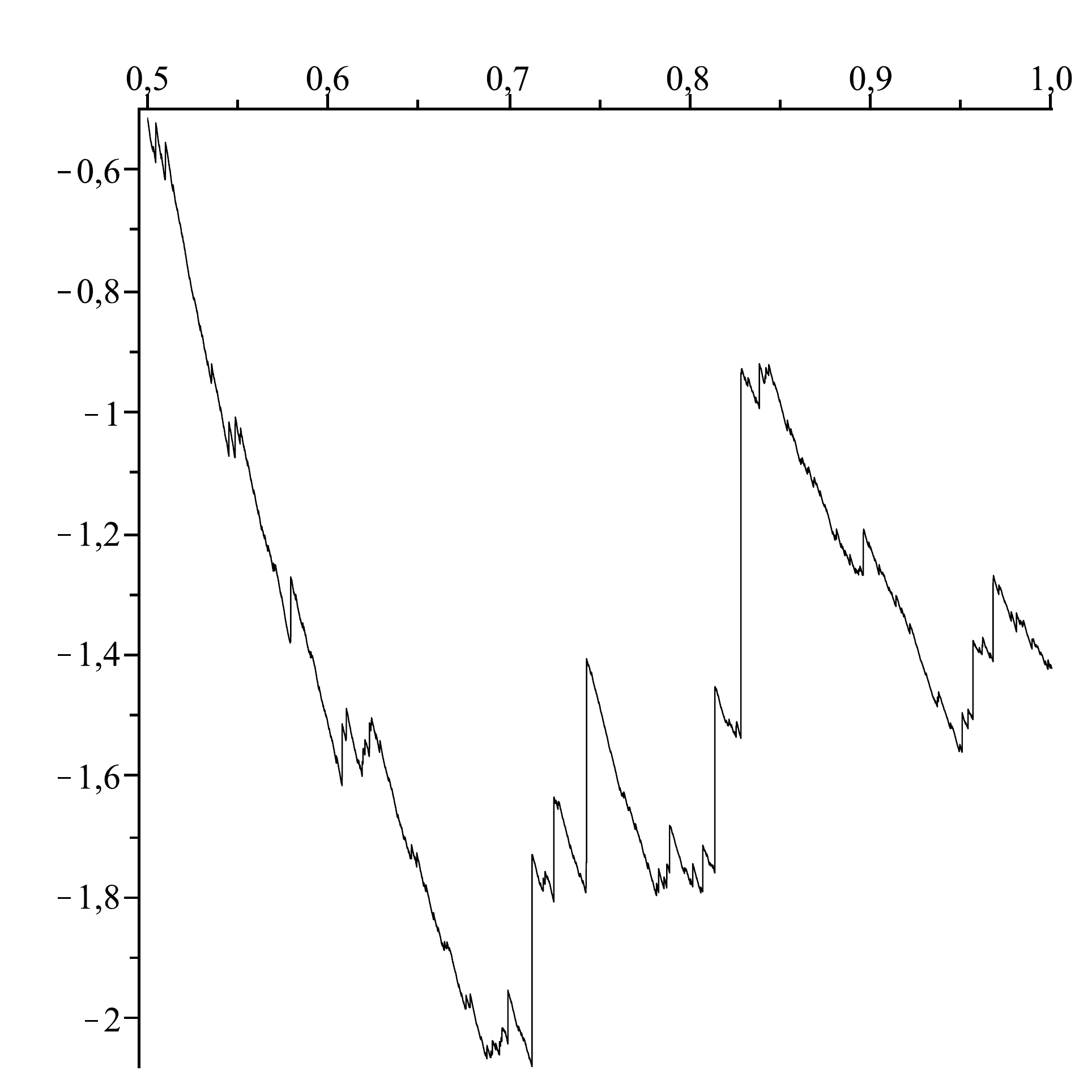}
\caption{\small Simulation of four approximations to the sample paths of $Y$. For better visibility the jumps are shown as vertical lines.}
\label{pb}
\end{figure}
Note that the sample paths do only have upward jumps due to the fact that the L\'evy measure $\phi$ is concentrated on $2^\ganz$.

The main goal of our paper is to determine the Hausdorff dimension of the range $Y([\frac12,1])=\{Y(t):\,t\in[\frac12,1]\}$ and the graph $G_Y([\frac12,1])=\{(t,Y(t)):\,t\in[\frac12,1]\}$ of the stochastic process $Y$ encoding all the possible distributional limits of St.\ Petersburg games. For an arbitrary subset $F\subseteq\rd$ the $s$-dimensional Hausdorff measure is defined as
$$
\mathcal{H}^s(F)=\lim_{\delta\rightarrow0}\inf\left\{\sum_{i=1}^{\infty}|F|_i^s: |F_i|\leq\delta \mbox{ and } F\subseteq\bigcup_{i=1}^\infty F_i\right\},
$$
where $|F|=\sup\{\|x-y\|:x,y\in F\}$ denotes the diameter of a set $F\subseteq\mathbb{R}^d$. It can now be shown that there exists a unique value $\dim_{\rm H} F\geq 0$ so that $\mathcal{H}^s(F)=\infty$ for all $0\leq s<\dim_{\rm H} F$ and $\mathcal{H}^s(F)=0$ for all $s>\dim_{\rm H} F$. This critical value is called the Hausdorff dimension of $F$. Specifically, we have
$$
\dim_{\rm H} F=\inf\left\{s:\mathcal{H}^s(F)=0\right\}=\sup\left\{s:\mathcal{H}^s(F)=\infty\right\}.
$$
For details on the Hausdorff dimension we refer to \cite{Fal3, Mat}.

An alternative fractal dimension is the so called box-counting dimension (see, e.g., \cite{Fal3}). For this purpose let $N_\delta(F)$ be the smallest number of closed balls of radius $\delta$ that cover the set $F\subseteq\rd$. The lower and the upper box-counting dimensions of an arbitrary set $F\subseteq\rd$ are now defined as
$$
 \diml_{\rm B} F=\liminf_{\delta\rightarrow0}\frac{\log N_\delta(F)}{-\log\delta}\quad\text{ and }\quad
 \dimu_{\rm B} F=\limsup_{\delta\rightarrow0}\frac{\log N_\delta(F)}{-\log\delta}
$$
and the box-counting dimension of $F$ is given by
\begin{align*}
 \dim_{\rm B} F=\lim_{\delta\rightarrow0}\frac{\log N_\delta(F)}{-\log\delta}
\end{align*}
provided that this limit exists. The different fractal dimensions are related as follows:
\begin{equation}\label{dimrel}
\dim_{\rm H} F\leq\diml_{\rm B}F\leq\dimu_{\rm B} F\leq d.
\end{equation}
Note that there are plenty of sets $F\subseteq\rd$ where these inequalities are strict.

In Section 2 we will determine the Hausdorff and box-counting dimension of the range $Y([\frac12,1])$ and the graph $G_Y([\frac12,1])$ for almost all sample paths of the stochastic process $Y$. Additionally, in Section 3  we will also consider a deterministic sequence introduced by Steinhaus \cite{Ste} which is called the ``{\it Steinhaus sequence}'' according to \cite{CS1}. The Steinhaus sequence $(x_n)_{n\in\nat}$ is defined by $x_n=2^k$ if $n=2^{k-1}+m\cdot 2^k$ for some $k\in\nat$ and $m\in\nat_0$. Alternatively, as in Vardi \cite{Var2}, one can define $x_n$ to be twice the highest power of $2$ dividing $n$. The Steinhaus sequence is explicitly given by
$$2,4,2,8,2,4,2,16,2,4,2,8,2,4,2,32,2,4,2,8,2,4,2,16,2,4,2,8,2,4,2,64,\ldots$$ 
and has relative frequencies $\lim_{n\to\infty}n^{-1}\card\{1\leq j\leq n:\,x_j=2^k\}=2^{-k}$ for $k\in\nat$. The sequence $(x_n)_{n\in\nat}$ has been considered as time-dependent entrance fees for repeated St.\ Petersburg games in \cite{Ste,CS1} and has been proven to be a sequence of nearly asymptotically fair premiums in a certain sense. For details we refer to \cite{CS1}. In contrast to \cite{Ste,CS1} we will consider the Steinhaus sequence as a sequence of possible gains in repeated St.\ Petersburg games.
Again, we will determine the Hausdorff and box-counting dimension of the range and the graph of the specific sample path of $Y$ resulting as a limiting object of the Steinhaus sequence. To do so, we will employ results for iterated function systems as presented in \cite{Fal4}.

\section{Hausdorff dimension of the St.\ Petersburg game}

\subsection{Hausdorff dimension of the range}

In this section we evaluate the Hausdorff dimension of the range of the stochastic process $Y=\{Y(t)\}_{t\in[\frac{1}{2},1]}$. We employ common techniques used to calculate Hausdorff dimensions of selfsimilar L\'evy processes (see \cite{Xiao,MX,KW}) and adapt them to our situation. Note that the given process $Y$ is neither a L\'evy process nor does it have the selfsimilarity property of a semistable process. The result is stated in the theorem below.
\begin{theorem}\label{dimrangeR}
 We have $\dim_{\rm H} Y([\frac{1}{2},1])=1$ almost surely.
\end{theorem}
Note that Theorem \ref{dimrangeR} together with \eqref{dimrel} yields $\dim_{\rm H} Y([\frac{1}{2},1])=\dim_{\rm B} Y([\frac{1}{2},1])=1$ almost surely. Since $Y$ is a process on $\rr$ it is obvious that $\dim_{\rm H} Y([\frac{1}{2},1])\leq1$ almost surely. For the proof of Theorem \ref{dimrangeR} it is hence sufficient to prove the following lemma.
\begin{lemma}\label{dimrangeRlow}
 We have $\dim_{\rm H} Y([\frac{1}{2},1])\geq1$ almost surely.
\end{lemma}
\begin{proof}
 As mentioned above we can write
\begin{align*}
 Y(t)=t^{-1}\left(X(t)-t\log_2t\right),
\end{align*}
where $X=\{X(t)\}_{t\geq0}$ is a semistable L\'evy process. To prove the proposition we will apply Frostman's theorem \cite{Kah,Mat} with the probability measure $\sigma=2\lambda|_{[\frac{1}{2},1]}$, where $\lambda$ denotes Lebesgue measure.  For this purpose let $0<\gamma<1$ and note that $\sigma$ is an admissible measure for Frostman's lemma, i.e.
$$\int_{\frac12}^1\int_{\frac12}^1\frac{\sigma(ds)\,\sigma(dt)}{|s-t|^\gamma}<\infty.$$
 By Frostman's theorem it is now sufficient to show that
\begin{equation}\label{Frostman}
 \int_{\frac{1}{2}}^{1}\int_{\frac{1}{2}}^{1}\Exp\left[|Y(s)-Y(t)|^{-\gamma}\right]ds\,dt <\infty.
\end{equation}
For $r\in[\frac12,1]$ let $g_r$ be a Lebesgue density of $X(r)$ chosen from the class $C^\infty(\rr)$ by Proposition 2.8.1 in \cite{Sat}. Then we have $M:=\sup_{r\in[\frac{1}{2},1]}\sup_{x\in\rr}\left|g_r(x)\right|<\infty$ as in Lemma 2.2 of \cite{KW}. By symmetry of the integrand we get
\begin{align*}
&\int_{\frac{1}{2}}^{1}\int_{\frac{1}{2}}^{1}\Exp\left[|Y(s)-Y(t)|^{-\gamma}\right]ds\,dt\\
&= 2 \int_{\frac{1}{2}}^{1}\int_{\frac{1}{2}}^{t}\Exp\left[\left|s^{-1}X(s)-\log_2s-t^{-1}\left(X(s)+(X(t)-X(s)))+\log_2t\right)\right|^{-\gamma}\right]ds\,dt\\
&= 2 \int_{\frac{1}{2}}^{1}\int_{\frac{1}{2}}^{t}\int_{\rr^2}\left|s^{-1}x-\log_2s-t^{-1}(x+y)+\log_2t\right|^{-\gamma}g_s(x)\,g_{t-s}(y)\,d\lambda^2(x,y)\,ds\,dt\\
&= 2 \int_{\frac{1}{2}}^{1}\int_{\frac{1}{2}}^{t}\int_{\rr^2}\left|\frac{t-s}{st}\,x+\log_2\left(\frac{t}{s}\right)-\frac{y}{t}\right|^{-\gamma}g_s(x)\,g_{t-s}(y)\,d\lambda^2(x,y)\,ds\,dt\\
&= 2 \int_{\frac{1}{2}}^{1}\int_{0}^{t-\frac{1}{2}}\int_{\rr^2}\left|\frac{w}{t(t-w)}\,x+\log_2\left(\frac{t}{t-w}\right)-\frac{y}{t}\right|^{-\gamma}g_{t-w}(x)\,g_w(y)\,d\lambda^2(x,y)\,dw\,dt,
\end{align*}
where in the last equality we substituted $w=t-s$. Now we write $w\in[0,\frac{1}{2}]$ as $w=2^{-m}r$ with $m=m(w)\in\nat$ and $r\in(\frac{1}{2},1]$. This leads us to
\begin{align*}
 g_w(y)&=\frac{d}{dy}\mathbb{P}\left(X(w)\leq y\right) = \frac{d}{dy}\mathbb{P}\left(X(2^{-m}r)\leq y\right) = \frac{d}{dy}\mathbb{P}\left(2^{-m}(X(r)-mr)\leq y\right)\\
&= \frac{d}{dy}\mathbb{P}\left(X(r)\leq 2^my+mr\right) = 2^mg_r\left(2^my+mr\right).
\end{align*}
Using the substitutions $v=2^my+mr$ and $u=\frac{t}{2^{-m}}\left(\frac{w}{t(t-w)}x+\log_2\left(\frac{t}{t-w}\right)+\frac{mw}{t}\right)$ we get
\begin{align*}
 &\int_{\rr^2}\left|\frac{w}{t(t-w)}\,x+\log_2\left(\frac{t}{t-w}\right)-\frac{y}{t}\right|^{-\gamma}g_{t-w}(x)\,g_w(y)\,d\lambda^2(x,y)\\
 &= 2^m\int_{\rr^2}\left|\frac{w}{t(t-w)}\,x+\log_2\left(\frac{t}{t-w}\right)-\frac{y}{t}\right|^{-\gamma}g_{t-w}(x)\,g_r(2^my+mr)\,d\lambda^2(x,y)\\
 &= \int_{\rr^2}\left|\frac{w}{t(t-w)}\,x+\log_2\left(\frac{t}{t-w}\right)-\frac{2^{-m}}{t}\,v+\frac{mw}{t}\right|^{-\gamma}\,g_{t-w}(x)g_r(v)\,d\lambda^2(x,v)\\
 &= \frac{t-w}{r}\int_{\rr^2}\left|\frac{2^{-m}}{t}(u-v)\right|^{-\gamma}g_{t-w}(x(u))\,g_r(v)\,d\lambda^2(u,v)\\
 &= \frac{t^\gamma(t-w)2^{m\gamma}}{r}\left(\int_A+\int_{A^\complement}\right)\left|u-v\right|^{-\gamma}g_{t-w}(x(u))\,g_r(v)\,d\lambda^2(u,v),
\end{align*}
where $A$ denotes the set $A=\left\{(u,v)\in\rr^2:|u-v|\leq1\right\}$. We now estimate the two integrals separately.
First,
\begin{align*}
 &\int_A\left|u-v\right|^{-\gamma}g_{t-w}(x(u))\,g_r(v)\,d\lambda^2(u,v)\\
 &\quad\leq M\int_\rr\left(\int_{v-1}^v(v-u)^{-\gamma}du + \int_v^{v+1}(u-v)^{-\gamma}du\right)g_r(v)\,dv\\
 &\quad= M\int_\rr\frac{2}{1-\gamma}g_r(v)\,dv = \frac{2M}{1-\gamma}
\end{align*}
and secondly,
\begin{align*}
 &\int_{A^\complement}\left|u-v\right|^{-\gamma}g_{t-w}(x(u))\,g_r(v)\,d\lambda^2(u,v)\leq \int_{A^\complement}g_{t-w}(x(u))\,g_r(v)\,d\lambda^2(u,v)\\
 &\quad\leq \frac{r}{t-w}\int_{\rr^2}g_{t-w}(x)\,g_r(v)\,d\lambda^2(x,v)=\frac{r}{t-w}.
\end{align*}
This leads us to
\begin{align*}
 &\int_{\rr^2}\left|\frac{w}{t(t-w)}\,x+\log_2\left(\frac{t}{t-w}\right)-\frac{y}{t}\right|^{-\gamma}g_{t-w}(x)\,g_w(y)\,d\lambda^2(x,y)\\
 &\quad\leq t^\gamma2^{m\gamma}\left(\frac{2M(t-w)}{r(1-\gamma)}+1\right)\leq t^\gamma2^{m\gamma}\left(\frac{4M}{1-\gamma}+1\right)=:Kt^\gamma2^{m\gamma}.
\end{align*}
Taken all together, we obtain 
\begin{align*}
 &\int_{\frac{1}{2}}^{1}\int_{\frac{1}{2}}^{1}\Exp\left[|Y(s)-Y(t)|^{-\gamma}\right]ds\,dt \leq 2K\int_\frac{1}{2}^1\int_0^{t-\frac{1}{2}}t^\gamma2^{m(w)\gamma}\,dw\,dt\\
 &\quad=2K\int_0^\frac{1}{2}\int_{\frac{1}{2}+w}^1 t^\gamma2^{m(w)\gamma}\,dt\,dw =2K\sum_{m\in\nat}\int_{2^{-(m+1)}}^{2^{-m}}\int_{\frac{1}{2}+w}^1 t^\gamma2^{m\gamma}\,dt\,dw\\
 &\quad\leq 2K\sum_{m\in\nat}\int_{\frac{1}{2}}^{1}\int_\frac{1}{2}^1 t^\gamma2^{m\gamma}\,dt \;2^{-m}\,dr=K\sum_{m\in\nat}\left(2^{\gamma-1}\right)^m\int_{\frac{1}{2}}^{1}t^\gamma\,dt <\infty,
\end{align*}
since $\gamma-1<0$. This concludes our proof.
\end{proof}

\subsection{Hausdorff dimension of the graph}

In this section we show that the dimension result for the range of the stochastic process $Y$ also holds for its graph $G_Y([\frac{1}{2},1])$. We will split the proof into two parts, first verifying $\alpha=1$ as an upper bound and secondly as a lower bound for the Hausdorff dimension of the graph.\\

We first calculate the upper bound for the Hausdorff dimension of the graph of the semistable L\'evy process $X$ and later on transfer the result to the process $Y$. As $X$ is not strictly semistable we can't use the dimension results of \cite{KW}, without modifying it according to our situation. Note that for the strictly stable (symmetric) Cauchy process on $\rr$ it is known that the Hausdorff dimension of the range coincides with those of an asymmetric (non-strictly) stable Cauchy process; see \cite{BG,Haw,Tay}.

\begin{theorem}\label{dimXup}
 Let $\{Z(t):=(t,X(t))\}_{t\geq0}$. Then almost surely
\begin{align*}
 \dim_{\rm H} Z([\tfrac{1}{2},1])\leq 1.
\end{align*}
\end{theorem}

Let $T(a,s)=\int_0^s1_{B(0,a)}(Z(t))\,dt$ denote the sojourn time of the L\'evy process $Z$ up to time $s$ in the closed ball $B(0,a)\subseteq\rr^2$ with radius $a$ centered at the origin. To prove Theorem \ref{dimXup} we need the following lemma. 

\begin{lemma}\label{sojour}
 Let $Z$ be the stochastic process from above. There exists a positive and finite constant $K$ such that for all $0<a\leq1$ and $\frac{a}{\sqrt{2}}\leq s\leq1$ we have
\begin{align*}
 \Exp\left[T(a,s)\right]\geq Ka.
\end{align*}
\end{lemma}

\begin{proof}
 Fix $0<a\leq1$ and let $0<\delta\leq\frac1{\sqrt{2}}$, to be specified later, so that
\begin{align*}
 2^{-(i_0 +1)}<a\delta\leq 2^{-i_0}<\frac{a}{\sqrt{2}}<s, 
\end{align*}
for some $i_0\in\nat_0$. We have
\begin{align*}
 &\Exp\left[T(a,s)\right]=\int_{0}^{s}\mathbb{P}\left(\|Z(t)\|<a\right)dt \geq \int_{0}^{s}\mathbb{P}\left(|X(t)|<\frac{a}{\sqrt{2}}, t<\frac{a}{\sqrt{2}}\right)dt\\
&= \int_{0}^{\frac{a}{\sqrt{2}}}\mathbb{P}\left(|X(t)|<\frac{a}{\sqrt{2}}\right)dt \geq \int_{0}^{2^{-i_0}}\mathbb{P}\left(|X(t)|<\frac{a}{\sqrt{2}}\right)dt\\
& = \sum_{i=i_0 +1}^{\infty}\int_{2^{-i}}^{2^{-i+1}}\mathbb{P}\left(|X(t)|<\frac{a}{\sqrt{2}}\right)dt = \sum_{i=i_0 +1}^{\infty}2^{-i}\int_{1}^{2}\mathbb{P}\left(|X(2^{-i}r)|<\frac{a}{\sqrt{2}}\right)dr\\
&= \sum_{i=i_0 +1}^{\infty}2^{-i}\int_{1}^{2}\mathbb{P}\left(|2^{-i}(X(r)-ir)|<\frac{a}{\sqrt{2}}\right)dr\\
&= \sum_{i=i_0 +1}^{\infty}2^{-i}\int_{1}^{2}\mathbb{P}\left(|X(r)-ir|<\frac{2^ia}{\sqrt{2}}\right)dr.
\end{align*}
The probability from above can be estimated from below by
\begin{align*}
 &\mathbb{P}\left(|X(r)-ir|<\frac{2^ia}{\sqrt{2}}\right) = \mathbb{P}\left(-\frac{a}{\sqrt{2}}2^i+ir<X(r)<\frac{a}{\sqrt{2}}2^i+ir\right)\\
&\quad\geq \mathbb{P}\left(-\frac{a}{\sqrt{2}}2^i+2i<X(r)<\frac{a}{\sqrt{2}}2^i\right)\\
&\quad= \mathbb{P}\left(X(r)<\frac{a}{\sqrt{2}}2^i\right)-\mathbb{P}\left(X(r)\leq -\frac{a}{\sqrt{2}}2^i+2i\right)\\
&\quad\geq \mathbb{P}\left(\sup_{r\in[1,2)}X(r)<\frac{a}{\sqrt{2}}2^i\right)-\mathbb{P}\left(\inf_{r\in[1,2)}X(r)\leq -\frac{a}{\sqrt{2}}2^i+2i\right)\\
& \quad\geq \mathbb{P}\left(\sup_{r\in[1,2)}X(r)<\frac{a}{2\sqrt{2}}2^{i_0+1}\right)-\mathbb{P}\left(\inf_{r\in[1,2)}X(r)\leq -\frac{a}{2\sqrt{2}}2^{i_0+1}\right)
\end{align*}
if we choose $i_0\in\nat_0$ large enough so that $2i\leq\frac{a}{2\sqrt{2}}2^i$ for all $i>i_0$.
As $X$ is a L\'evy process, we can assume that it has c\`adl\`ag paths and thus both $\sup_{r\in[1,2)}X(r)$ and $\inf_{r\in[1,2)}X(r)$ are random variables. Hence we can choose $0<\delta\leq\frac1{\sqrt{2}}$ from above small enough (i.e., $i_0\in\nat_0$ even bigger) so that we have
\begin{align*}
 &\mathbb{P}\left(\sup_{r\in[1,2)}X(r)<\frac{a}{2\sqrt{2}}2^{i_0+1}\right)-\mathbb{P}\left(\inf_{r\in[1,2)}X(r)\leq -\frac{a}{2\sqrt{2}}2^{i_0+1}\right)\\
& \quad \geq\mathbb{P}\left(\sup_{r\in[1,2)}X(r)<\frac{1}{\delta 2\sqrt{2}}\right)-\mathbb{P}\left(\inf_{r\in[1,2)}X(r)\leq -\frac{1}{\delta 2\sqrt{2}}\right)\geq\frac12.
\end{align*}
Note that $\delta$ does not depend on $a$. It follows that
\begin{align*}
 \Exp\left[T(a,s)\right]\geq \sum_{i=i_0 +1}^{\infty}2^{-i}\int_{1}^{2}\frac12\,dr =\frac12\sum_{i=i_0 +1}^{\infty}2^{-i} = \frac12 2^{-i_0}\geq \frac12\delta a=:Ka,
\end{align*} 
which concludes the proof.
\end{proof}

\begin{proof}[Proof of Theorem \ref{dimXup}]
Let $K_1>0$ be a fixed constant. A family $\Lambda(a)$ of cubes of side $a\in(0,1]$ in $\rr^2$ is called $K_1$-nested if no ball of radius $a$ in $\rr^2$ can intersect more than $K_1$ cubes of $\Lambda(a)$. For any $u\geq0$ let $M_u(a,s)$ be the number of cubes hit by the L\'evy process $Z$ at some time $t\in[u,u+s]$. Then a famous covering lemma of Pruitt and Taylor \cite[Lemma 6.1]{PT} states that 
$$\Exp[M_u(a,s)]\leq2K_1s\cdot(\Exp[T(\tfrac{a}3,s)])^{-1}.$$
Lemma \ref{sojour} now enables us to construct a covering of $Z([\frac12,1])$ whose expected $s$-dimensional Hausdorff measure is finite for every $s>1$. The arguments are in complete analogy to the proof of part (i) of Lemma 3.4 in \cite{KW} and thus omitted.
\end {proof}

In order to transfer the result of Theorem \ref{dimXup} to the process $Y$ we can now write all elements $(t,Y(t))^\top\in G_Y([\frac{1}{2},1])$ as
\begin{align*}
 \left(\begin{array}{c} t \\ Y(t) \end{array}\right)=\left(\begin{array}{c} t \\ t^{-1}X(t)-\log_2t \end{array}\right)=: T\left(t,X(t)\right).
\end{align*}
It can easily be shown that for a fixed constant $C>0$ the function
\begin{align*}
 T:[\tfrac{1}{2},1]\times[-C,C]\rightarrow \Ima(T), \quad\left(\begin{array}{c} t \\ x \end{array}\right)\mapsto T(t,x) = \left(\begin{array}{c} t \\ t^{-1}x-\log_2t \end{array}\right)
\end{align*}
is bi-Lipschitz. Since $X$ is a L\'evy process, it can be assumed that all paths are c\'adl\'ag and hence that for all fixed $\omega\in\Omega$ there exists a constant $C(\omega)<\infty$ such that
\begin{align*}
 X(t)(\omega)\in[-C(\omega),C(\omega)] \quad \text{ for all } t\in[\tfrac{1}{2},1].
\end{align*}
This means that for $Z=\left(Z(t)=(t,X(t))\right)_{t\in[\frac{1}{2},1]}$ and all $\omega\in\Omega$ we have
\begin{align*}
 \dim_{\rm H} Z([\tfrac{1}{2},1])(\omega)=\dim_{\rm H} T(Z([\tfrac{1}{2},1]))(\omega)=\dim_{\rm H} G_Y([\tfrac{1}{2},1])(\omega)
\end{align*}
by Lemma 1.8 in \cite{Fal1}. Since we have shown in Theorem \ref{dimXup} that $\dim_{\rm H} Z([\frac{1}{2},1])\leq1$ almost surely, we have thus proven the following upper bound.

\begin{theorem}\label{dimHGup}
We have $\dim_{\rm H} G_Y([\frac{1}{2},1])\leq1$ almost surely.
\end{theorem}

To prove the lower bound for the Hausdorff dimension of the graph we can use the same technique as for the lower bound in case of the range of $Y$.

\begin{theorem}\label{dimHGlow}
 We have $\dim_{\rm H} G_Y([\frac{1}{2},1])\geq1$ almost surely.
\end{theorem}

\begin{proof}
Let $0<\gamma<1$. By \eqref{Frostman} we get
\begin{align*}
&\int_{\frac{1}{2}}^{1}\int_{\frac{1}{2}}^{1}\Exp\left[\|(s,Y(s))^\top-(t,Y(t))^\top\|^{-\gamma}\right]ds\,dt\\
&\quad= \int_{\frac{1}{2}}^{1}\int_{\frac{1}{2}}^{1}\Exp\left[\left((s-t)^2+(Y(s)-Y(t))^2\right)^{-\frac{\gamma}{2}}\right]ds\,dt\\
&\quad\leq \int_{\frac{1}{2}}^{1}\int_{\frac{1}{2}}^{1}\Exp\left[|Y(s)-Y(t)|^{-\gamma}\right]ds\,dt<\infty.
\end{align*}
The rest of the proof runs exactly as the proof of Lemma \ref{dimrangeRlow}.
\end{proof}

With similar techniques it is also possible to proof the following dimension result for the box-counting dimension of the graph of the St.\ Petersburg process $Y$.

\begin{theorem}\label{dimBG}
 We have $\dim_{\rm B}G_Y([\frac{1}{2},1])=1$ almost surely.
\end{theorem}

\begin{proof}
The lower bound follows directly from the almost sure inequalities
$$
1\leq\dim_{\rm H}G_Y([\tfrac{1}{2},1])\leq\diml_{\rm B}G_Y([\tfrac{1}{2},1])\leq\dimu_{\rm B}G_Y([\tfrac{1}{2},1]).
$$
For the upper bound it is now sufficient to verify $\dimu_{\rm B}G_Y([\frac{1}{2},1])\leq1$ almost surely.
With similar arguments as in the proof of Theorem \ref{dimXup} we can show that $\dimu_{\rm B}Z([\frac{1}{2},1])\leq1$ almost surely; see also the proof of Lemma 3.5 in \cite{MX}.
With the bi-Lipschitz invariance of the upper box-counting dimension (see section 3.2 in \cite{Fal3}) the proof concludes.
\end{proof}

\begin{remark}
If one prefers to flip an unfair coin this naturally leads to so called generalized St.\ Petersburg games as treated in \cite{CK,Gut,Pap}. Let $p\in(0,1)$ be the probability of the coin falling heads and let $q=1-p$. Then a gain of $p^{-1}q^{1-T}$ in a single St.\ Petersburg game results in the limit theorem
$$\frac{S_{k(n)}-k(n)\log_{1/q}k(n)}{k(n)}\to Y(t)=t^{-1}(X(t)-t\log_{1/q}t)$$
in distribution, whenever
$$q^{\lceil \log_{1/q}k(n)\rceil}k(n)\to t\in[q,1],$$
where $\{X(t)\}_{t\geq0}$ is a semistable L\'evy process with the semi-selfsimilarity property
$$X(q^kt)\eqd q^k(X(t)+kt)\quad\text{ for every }k\in\ganz\text{ and }t\geq0.$$
We emphasize that with the above techniques our Theorems \ref{dimrangeR}, \ref{dimHGup}, \ref{dimHGlow} and \ref{dimBG} also hold for the process $\{Y(t)\}_{t\in[q,1]}$ in this generalized situation when replacing the interval by $[q,1]$.
\end{remark}

\section{Hausdorff dimension of the Steinhaus sequence}

Recall the definition of the Steinhaus sequence $(x_n)_{n\in\nat}$ given in the Introduction. The asymptotic properties of $(x_n)_{n\in\nat}$ have been analyzed in full detail by Cs\"org\H{o} and Simons \cite{CS1}. Let $s(n)=x_1+\cdots+x_n$ and $\gamma_n=n\cdot 2^{-\lceil\log_2n\rceil}\in(\frac12,1]$ then by Theorem 3.3 in \cite{CS1} we have for any $n\in\nat$
\begin{equation}\label{asymptSs}
\frac{s(n)-n\log_2n}{n}=\xi(\gamma_n),
\end{equation}
where the function $\xi:[\frac12,1]\to[0,2]$ is defined by
$$\xi(\gamma)=2-\log_2\gamma-\frac1{\gamma}\sum_{k=1}^\infty\frac{k\varepsilon_k}{2^k}$$
and the sequence $(\varepsilon_k)_{k\in\nat}\subseteq\{0,1\}$ is given by the dyadic expansion $\gamma=\sum_{k=0}^\infty\frac{\varepsilon_k}{2^k}$ of $\gamma\in[\frac12,1]$ with the convention that $\varepsilon_k=0$ for infinitely many $k\in\nat$. By Theorem 3.1 in \cite{CS1} the function $\xi$ is c\`adl\`ag with $\xi(\frac12)=2=\xi(1)$ and has jumps precisely at the dyadic rationals in $(\frac12,1]$. All these jumps are upward and the largest jump occurs from $\xi(1-)=0$ to $\xi(1)=2$. The graph of $\xi$ seems to inhere fractal properties as can be seen in Figure \ref{gxi} below, a replication of Figure 1 in \cite{CS1}. 
\begin{figure}[h]
\includegraphics[scale=0.5]{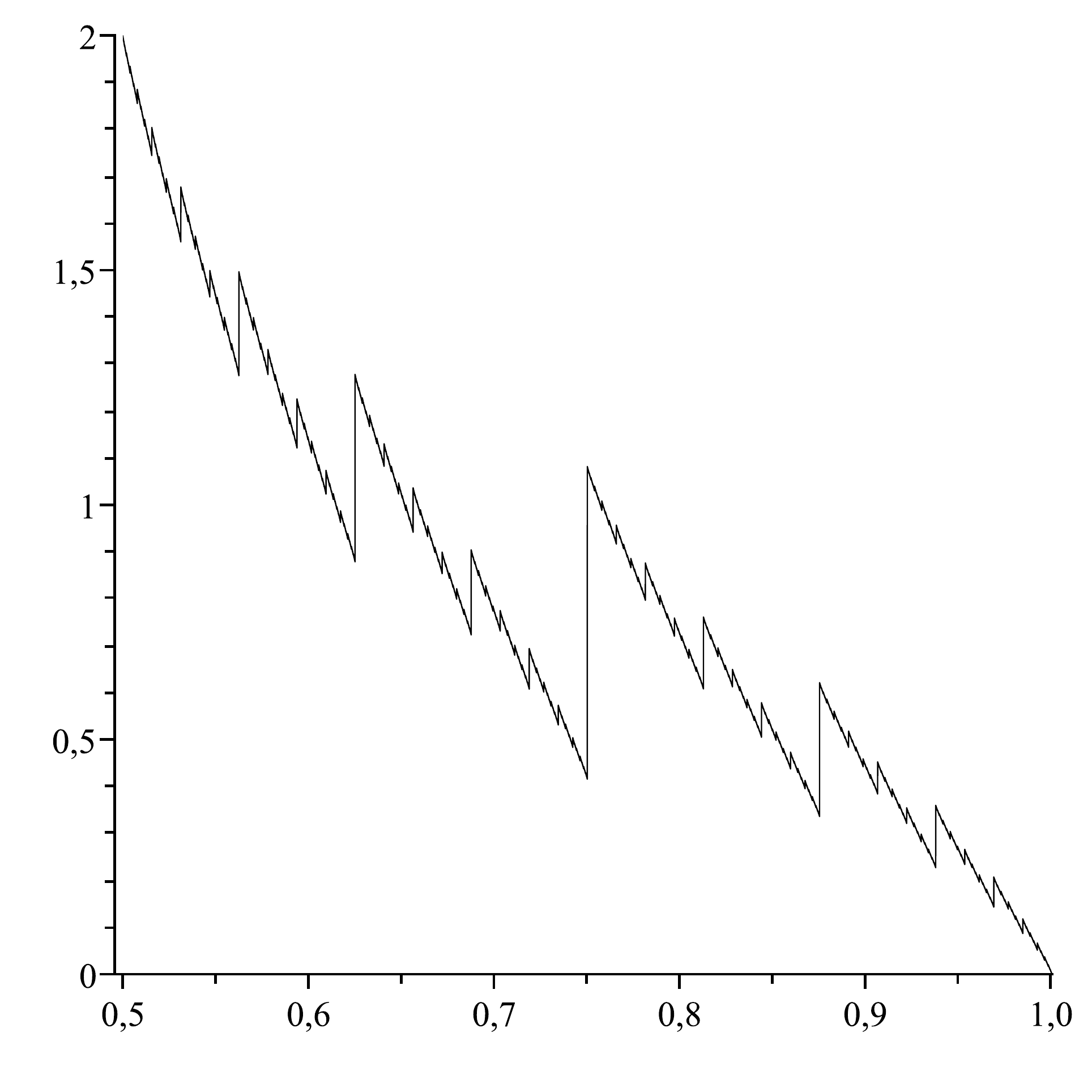}
\caption{\small Graph of $\xi$ on the interval $[\frac12,1)$. For better visibility the jumps of $\xi$ are shown as vertical lines.}
\label{gxi}
\end{figure}
It follows directly from \eqref{asymptSs} that the sequence $(s(n))_{n\in\nat}$ of total gains satisfies the asymptotic property of Feller
\begin{equation}\label{FellerSs}
\frac{s(n)}{n\log_2n}\to1
\end{equation}
as $n\to\infty$; see \cite{CS1}. Moreover, for any sequence $k_n\to\infty$ with  $k_n\cdot 2^{-\lceil\log_2k_n\rceil}\to\gamma\in[\frac12,1]$ we get from \eqref{asymptSs}
\begin{equation*}
\LIM\left\{\frac{s(k_n)-k_n\log_2k_n}{k_n}:\,n\in\nat\right\}=\{\xi(\gamma), \xi(\gamma-)\},
\end{equation*}
where $\LIM$ denotes the set of accumulation points. Hence we may consider the function $\xi$ as the corresponding limiting sample path of $\{Y(t)\}_{t\in[\frac12,1]}$. Note that \eqref{FellerSs} shows that the Steinhaus sequence is an exceptional sequence of gains when considering almost sure limit behavior, since Feller's law of large numbers does not hold in an almost sure sense. According to classical results in \cite{CR, Adl, CS2} it is known that
\begin{equation}\label{ACR}
\limsup_{n\to\infty}\frac{S_n}{n\log_2n}=\infty\quad\text{ and }\quad\liminf_{n\to\infty}\frac{S_n}{n\log_2n}=1\quad\text{ almost surely.}
\end{equation}
More precisely, by Corollary 1 in \cite{Var2} we have $\LIM\{S_n/(n\log_2n):\,n\in\nat\}=[1,\infty]$ almost surely,
but there is a version of the strong law of large numbers by \cite{CS3} when neglecting the largest gain
$$\frac{S_n-\max_{1\leq k\leq n}X_k}{n\log_2n}\to1\quad\text{ almost surely.}$$
A comparison of \eqref{FellerSs} and \eqref{ACR} shows that the Steinhaus sequence belongs to an exceptional nullset concerning almost sure limit behavior of the total gain in repeated St.\ Petersburg games. We will now show that the Steinhaus sequence is not exceptional concerning the local fluctuations of the limiting sample paths measured by the Hausdorff or box-counting dimension.

It follows directly from the above stated properties of $\xi$ given in Theorem 3.1 of \cite{CS1} that the range $\xi([\frac12,1])$ is equal to the interval $(0,2]$ and hence $\dim_{\rm H}\xi([\frac12,1])=1$ by Theorem 1.12 in \cite{Fal1}. This shows that $\dim_{\rm H}\xi([\frac12,1])$ coincides with the Hausdorff dimension of the range of a typical sample path of $\{Y(t)\}_{t\in[\frac12,1]}$. Clearly, by \eqref{dimrel} we also have $\dim_{\rm H}\xi([\frac12,1])=\dim_{\rm B}\xi([\frac12,1])=1$. A look at Figure \ref{gxi} suggests that it is merely the graph and not the range of $\xi$ that should inhere fractal properties. In the sequel we will argue that also the graph $G_\xi([\frac12,1])$ is typical concerning the almost sure dimension properties of the sample graph of $\{Y(t)\}_{t\in[\frac12,1]}$. 
To this aim we will again apply the bi-Lipschitz function $T$ from Section 2 whose inverse is given by $T^{-1}:[\frac12,1]\times[0,2]\to\Ima T^{-1}$ with $T^{-1}(t,x)=(t,t(x+\log_2t))^\top$. Applied to the graph of $\xi$ we get for any $\gamma\in[\frac12,1]$
$$T^{-1}(\gamma,\xi(\gamma))=\left(\begin{array}{c}
    \gamma \\ 
    \gamma(\xi(\gamma)+\log_2\gamma) \\ 
  \end{array}\right)=\left(\begin{array}{c}
    \gamma \\ 
    2\gamma-\sum_{k=1}^\infty\frac{k\varepsilon_k}{2^k} \\ 
  \end{array}\right)$$
and by bi-Lipschitz invariance we have
\begin{equation}\label{dimSs1}
\dim_{\rm H}G_\xi([\tfrac12,1])=\dim _{\rm H}T^{-1}(G_\xi([\tfrac12,1])).
\end{equation}
The same equality holds for upper and lower box-counting dimensions; e.g., see \cite{Fal3}.
The image $T^{-1}(G_\xi([\frac12,1)))$ is illustrated in Figure \ref{gTxi} and shows perfect selfsimilarity.
To see this, we may write $T^{-1}(\gamma,\xi(\gamma))=(\gamma,f(\gamma))^\top$ with $f(\gamma)=2\gamma-\sum_{k=1}^\infty\frac{k\varepsilon_k}{2^k}$.
\begin{figure}[h]
\includegraphics[scale=0.4]{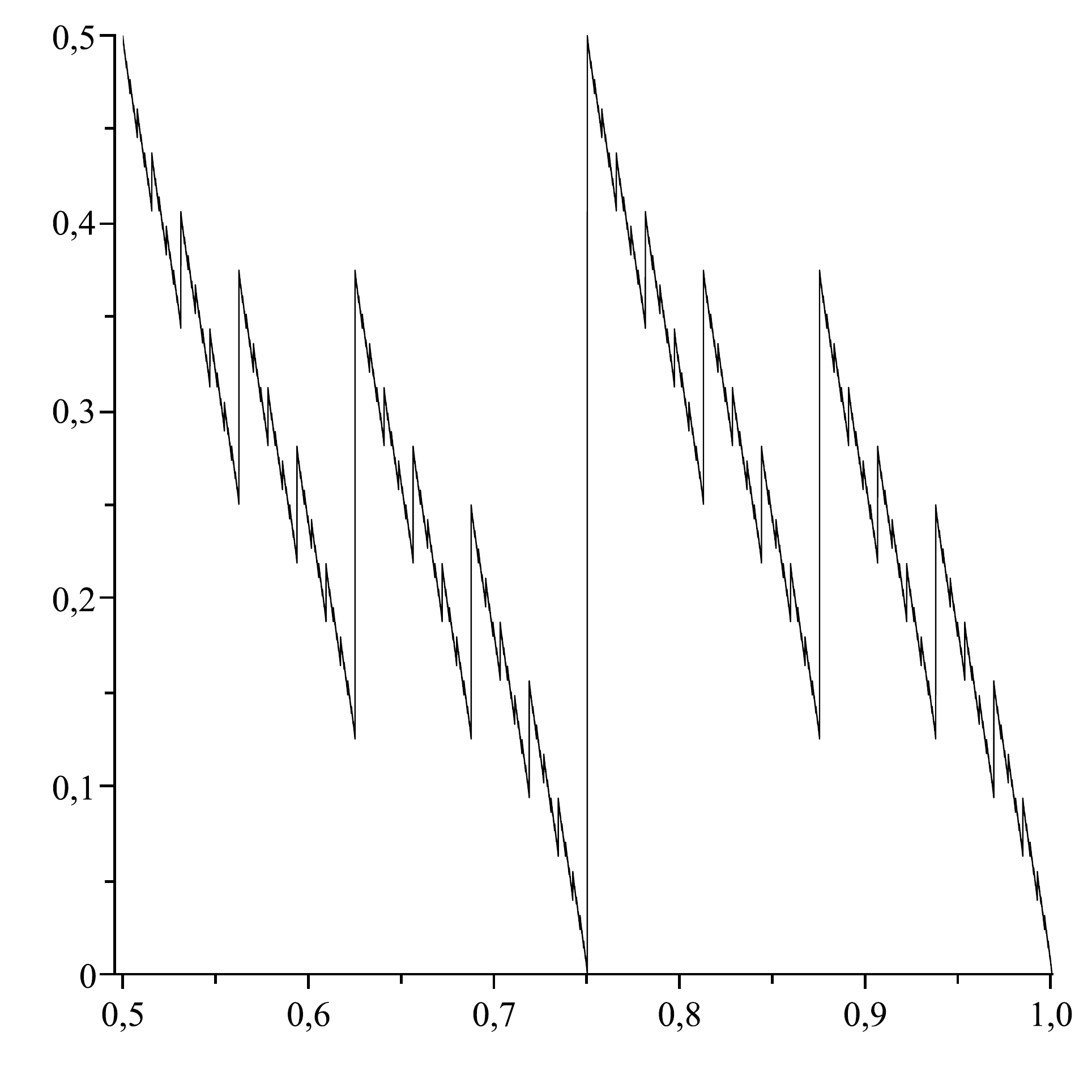}
\caption{\small Image of $T^{-1}(G_\xi[\frac12,1))$. For better visibility the jumps are shown as vertical lines.}
\label{gTxi}
\end{figure}
\begin{lemma}\label{selfsimSs}
Let $f$ be the function from above. Then for any $\gamma\in[\frac12,1)$ we have
$$f(\tfrac12\gamma+\tfrac12)=\tfrac12(1-\gamma+f(\gamma))=f(\tfrac12\gamma+\tfrac14).$$
\end{lemma}
\begin{proof}
For the dyadic expansion $\gamma=\sum_{k=1}^\infty\frac{\varepsilon_k}{2^k}$ of $\gamma\in[\frac12,1)$ we necessarily have $\varepsilon_1=1$. Consequently,
\begin{align*}
\tfrac12\gamma+\tfrac12=\tfrac12+\sum_{k=1}^\infty\frac{\varepsilon_k}{2^{k+1}}=\sum_{k=1}^\infty\frac{\varepsilon'_k}{2^k} & \quad\text{ with }\quad\varepsilon'_k=\begin{cases} 1 & k=1,\\ \varepsilon_{k-1} & k\geq2
\end{cases}\\
\intertext{and}
\tfrac12\gamma+\tfrac14=\tfrac14+\sum_{k=1}^\infty\frac{\varepsilon_k}{2^{k+1}}=\sum_{k=1}^\infty\frac{\varepsilon''_k}{2^k} & \quad\text{ with }\quad\varepsilon''_k=\begin{cases} 1 & k=1,\\ 0 & k=2,\\ \varepsilon_{k-1} & k\geq3.
\end{cases}
\end{align*}
It follows that 
\begin{align*}
f(\tfrac12\gamma+\tfrac12) & =2(\tfrac12\gamma+\tfrac12)-\sum_{k=1}^\infty\frac{k\varepsilon'_k}{2^k}=\gamma+\tfrac12-\sum_{k=2}^\infty\frac{k\varepsilon_{k-1}}{2^k}\quad\text{ and}\\
f(\tfrac12\gamma+\tfrac14) & =2(\tfrac12\gamma+\tfrac14)-\sum_{k=1}^\infty\frac{k\varepsilon''_k}{2^k}=\gamma-\sum_{k=3}^\infty\frac{k\varepsilon_{k-1}}{2^k}=\gamma+\tfrac12-\sum_{k=2}^\infty\frac{k\varepsilon_{k-1}}{2^k}.
\end{align*}
This shows $f(\tfrac12\gamma+\tfrac12)=f(\tfrac12\gamma+\tfrac14)=\gamma+\tfrac12-\sum_{k=2}^\infty\frac{k\varepsilon_{k-1}}{2^k}$ and furthermore we get
\begin{align*}
\gamma+\tfrac12-\sum_{k=2}^\infty\frac{k\varepsilon_{k-1}}{2^k} & =\gamma+\tfrac12-\tfrac12\sum_{k=1}^\infty\frac{(k+1)\varepsilon_{k}}{2^k}\\
& =\gamma+\tfrac12-\tfrac12\sum_{k=1}^\infty\frac{k\varepsilon_{k}}{2^k}-\tfrac12\gamma=\tfrac12(1-\gamma+f(\gamma))
\end{align*}
concluding the proof.
\end{proof}
Let $T_0,T_1:[\frac12,1]\times[0,\frac12]\to[\frac12,1]\times[0,\frac12]$ be the affine contractions given by
\begin{align*}
T_0(x,y) & =\left(  \begin{array}{c} \frac12 x+\frac14 \\ \frac12(1-x+y) \end{array}\right)
  =\left(  \begin{array}{cc}
    1/2 & 0 \\ 
    -1/2 &1/2
  \end{array}\right)\left(  \begin{array}{c}
    x \\ 
    y
  \end{array}\right)+
\left(  \begin{array}{c}
    1/4 \\ 
    1/2
  \end{array}\right),\\
T_1(x,y) & =\left(  \begin{array}{c} \frac12 x+\frac12 \\ \frac12(1-x+y) \end{array}\right)
  =\left(  \begin{array}{cc}
    1/2 & 0 \\ 
    -1/2 &1/2
  \end{array}\right)\left(  \begin{array}{c}
    x \\ 
    y
  \end{array}\right)+
\left(  \begin{array}{c}
    1/2 \\ 
    1/2
  \end{array}\right).
\end{align*}
Then it follows from Lemma \ref{selfsimSs} that for any $\gamma\in[\frac12,1)$
\begin{align*}
T^{-1}(\tfrac12\gamma+\tfrac14, \xi(\tfrac12\gamma+\tfrac14)) &  =T_0(\gamma,f(\gamma))=T_0\big(T^{-1}(\gamma,\xi(\gamma))\big)\quad\text{ and}\\
T^{-1}(\tfrac12\gamma+\tfrac12, \xi(\tfrac12\gamma+\tfrac12)) &  =T_1(\gamma,f(\gamma))=T_1\big(T^{-1}(\gamma,\xi(\gamma))\big).
\end{align*}
These contraction properties are illustrated in Figure \ref{congTxi} and show that the image $T^{-1}(G_\xi([\frac12,1)))$ can be generated by an iterated function system.
\begin{figure}[h]
\includegraphics[scale=0.35]{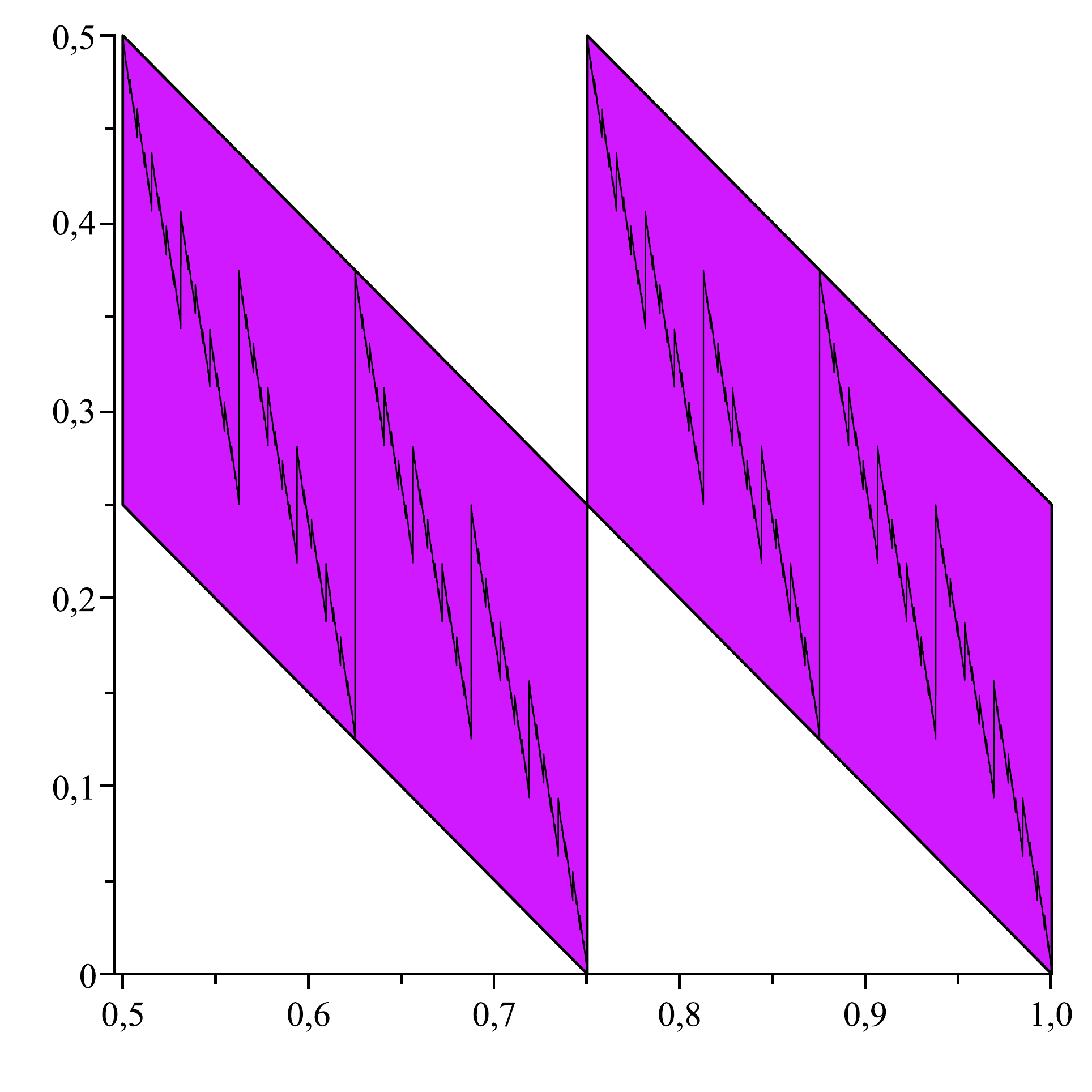}
\includegraphics[scale=0.35]{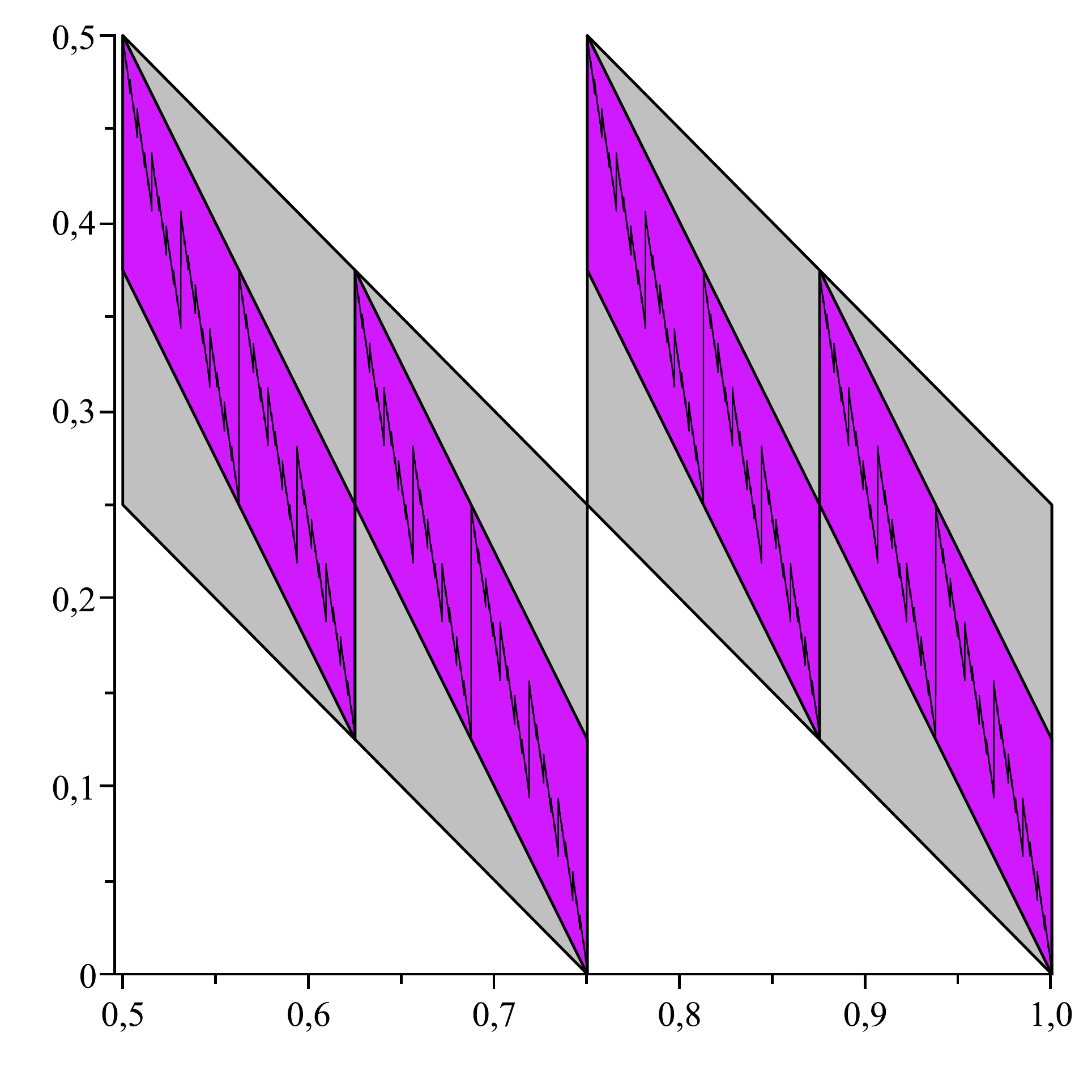}
\caption{\small Contractions generating the image (left) and their first iterates (right).}
\label{congTxi}
\end{figure}
By Hutchinson \cite{Hut} there exists a unique non-empty compact set $F\subseteq[\frac12,1]\times[0,\frac12]$, called the attractor, such that $F=T_0(F)\cup T_1(F)$ which fulfills 
$$F=\bigcap_{r=1}^\infty\,\,\bigcup_{(i_1,\ldots,i_r)\in\{0,1\}^r}T_{i_1}\circ\cdots\circ T_{i_r}([\tfrac12,1]\times[0,\tfrac12]).$$
Our construction shows that for $\gamma\in[\frac12,1)$ with dyadic expansion $\gamma=\sum_{k=1}^\infty\frac{\varepsilon_k}{2^k}$ we have $\varepsilon_1=1$ and
$$d\big(T_{\varepsilon_2}\circ\cdots\circ T_{\varepsilon_r}([\tfrac12,1]\times[0,\tfrac12]), T^{-1}(\gamma,\xi(\gamma))\big)\to0$$
as $r\to\infty$, where $d(A,x)=\inf\{\|y-x\|:\,y\in A\}$ for $A\subseteq\rr^2$ and $x\in\rr^2$. Since we required $\varepsilon_k=0$ for infinitely many $k\in\nat$, the only limit points missing are those with $T_{i_j}=T_1$ for all but finitely many $j\in\nat$. For these we have
$$d\big(T_{i_1}\circ\cdots\circ T_{i_r}([\tfrac12,1]\times[0,\tfrac12]), T^{-1}(\gamma,\xi(\gamma-))\big)\to0$$
for a dyadic rational $\gamma\in(\frac12,1]$. The above arguments show that $F$ is the closure of $T^{-1}(G_\xi([\frac12,1)))$ and since the dyadic rationals are countable, by elementary properties of the Hausdorff dimension and \eqref{dimSs1} we get
\begin{equation}\label{dimHBF}
\dim_{\rm H}F=\dim_{\rm H}T^{-1}(G_\xi([\tfrac12,1)))=\dim_{\rm H}G_\xi([\tfrac12,1]).
\end{equation}
The same equality holds for upper and lower box-counting dimensions; e.g., see \cite{Fal3}.

A common way to calculate the fractal dimension of the self-affine invariant set $F$ is by means of the singular value function. For on overview of such methods we refer to \cite{Fal4}. The linear part of both affine mappings $T_0$ and $T_1$ is equal to the linear contraction with associated matrix
$$L=\left(  \begin{array}{cc}
    1/2 & 0 \\ 
    -1/2 &1/2
  \end{array}\right).$$
By induction one easily calculates for $r\in\nat$
$$L^r=\left(  \begin{array}{cc}
    1/2^r & 0 \\ 
    -r/2^r &1/2^r
  \end{array}\right)$$
and the singular values of $L^r$ are the positive roots of the eigenvalues of $(L^r)^\top L^r$ which calculate as
\begin{equation}\label{singval}
\alpha_1^{(r)}=\frac1{2^r}\sqrt{\frac{r^2+2+\sqrt{r^4+4r^2}}{2}}\quad\text{ and }\quad\alpha_2^{(r)}=\frac1{2^r}\sqrt{\frac{r^2+2-\sqrt{r^4+4r^2}}{2}}.
\end{equation}
These determine the singular value function of $L^r$ for $r\in\nat$ given by
$$\varphi^s(L^r)=\begin{cases}
(\alpha_1^{(r)})^s & \text{ for }0<s\leq1,\\
\alpha_1^{(r)}(\alpha_2^{(r)})^{s-1} & \text{ for }1<s\leq2.
\end{cases}$$
Now the affinity dimension of $F$ is defined by
\begin{equation}\label{AdimF}
\dim_{\rm A}F=\inf\Big\{s>0:\,\sum_{r=1}^\infty2^r\varphi^s(L^r)<\infty\Big\}
\end{equation}
and the special form of the singular values in \eqref{singval} shows that $\dim_{\rm A}F=1$.

Since the union $F=T_0(F)\cup T_1(F)$ is disjoint, by Proposition 2 in�\cite{Fal2a} we get a lower bound for the Hausdorff dimension of $F$
\begin{equation}\label{afflowb}
\dim_{\rm H}F\geq\inf\Big\{s>0:\,\sum_{r=1}^\infty2^r\big(\varphi^s(L^{-r})\big)^{-1}<\infty\Big\}.
\end{equation}
Again, by induction one easily calculates for $r\in\nat$
$$L^{-r}=\left(  \begin{array}{cc}
    2^r & 0 \\ 
    r\,2^r & 2^r
  \end{array}\right)$$
and the singular values of $L^{-r}$ are 
\begin{equation*}
\beta_1^{(r)}=2^r\sqrt{\frac{r^2+2+\sqrt{r^4+4r^2}}{2}}\quad\text{ and }\quad\beta_2^{(r)}=2^r\sqrt{\frac{r^2+2-\sqrt{r^4+4r^2}}{2}},
\end{equation*}
which shows that $\dim_{\rm H}F\geq1$ by \eqref{afflowb}. Since by \cite{Fal2} we have
$$\dim_{\rm H}F\leq\diml_{\rm B}F\leq\dimu_{\rm B}F\leq\dim_{\rm A}F,$$
altogether the above calculations show:
\begin{theorem}
We have $\displaystyle\dim_{\rm H}G_\xi([\tfrac12,1])=1=\dim_{\rm B}G_\xi([\tfrac12,1]).$
\end{theorem}
This shows that the graph of $\xi$, being the limiting object of the Steinhaus sequence (considered as a possible sequence of total gains in repeated St.\ Petersburg games), is not exceptional concerning the Hausdorff or box-counting dimension of the sample graph $G_Y([\frac12,1])$ calculated in Section 2.

\bibliographystyle{plain}

\end{document}